\documentclass[12pt]{article}

\usepackage{amsmath, amsthm, amssymb, enumitem}
\usepackage{fullpage}
\usepackage{tikz}
\usepackage{tikz-qtree}

\title{\bf Structure and Growth of Galileo Sequences}
\author{William Cheah and David Treeby}
\date{}

\newtheorem{theorem}{Theorem}
\newtheorem{lemma}[theorem]{Lemma}
\newtheorem{definition}[theorem]{Definition}
\newtheorem{example}[theorem]{Example}

\begin{document}

\maketitle

\begin{abstract}
A \emph{Galileo sequence} \((a_n)\) is a sequence of positive integers whose partial sums $S_n$ satisfy $S_{2n}=kS_n$ for some $k>1$. In this paper we prove that every polynomial Galileo sequence is given by first differences of the form \(a_n= C\left(n^d-(n-1)^d\right)\). We then show that every positive Galileo sequence has a binary-tree representation. Finally, for positive monotone integer-valued Galileo sequences, we prove power-law growth bounds, and give a continuous analog together with a characterization of all continuous solutions.
\end{abstract}



\section{Introduction: Galileo sequences}

In 1615 Galileo \cite{Galileo} observed that the distance traveled by a falling body is proportional to the square of the time, a point presented in his 1638 treatise \emph{Two New Sciences}. Normalizing units so that the total distance traveled after \(n\) seconds is \(S_n=n^2\), we have
\(
S_n-S_{n-1}=n^2-(n-1)^2=2n-1,
\)
so the distance traveled during the \(n\)th second is \(2n-1\). Thus the successive one-second increments form the sequence of odd numbers, and their partial sums are the square numbers:
\(
S_n=1+3+\cdots+(2n-1)=n^2.
\)
It therefore follows that
\[
S_{2n}=(2n)^2=4S_n,
\]
for all \(n\ge 1\). This simple dilation law motivates the following definition.

\begin{definition}
Let \((a_n)\) be a sequence of positive integers with partial sums
\(
S_n=\sum_{i=1}^n a_i.
\)
We say that \((a_n)\) is a \textbf{Galileo sequence} if there exists an integer \(k>1\) such that
\begin{equation}\label{G}
S_{2n}=k\,S_n,
\tag{G}
\end{equation}
for every \(n\ge1\). We refer to \eqref{G} as the \textbf{Galileo relation}.
\end{definition}

Earlier work has considered such sequences. Pulapaka~\cite{Pulapaka} studied alternating Galileo sequences, and E\u{g}ecio\u{g}lu~\cite{Egecioglu} investigated iterated forms of the relation. Zeitlin~\cite{Zeitlin} stated that the polynomial solutions of the Galileo relation are exactly the family
\[
a_n=n^d-(n-1)^d.
\]
However, his note contains no proof of this claim, and we have not found a proof elsewhere in the literature. Therefore we include a short proof here. We also show that every positive Galileo sequence has a binary-tree representation. Later, we prove that all positive monotone integer-valued Galileo sequences satisfy a power-law growth bound, and conclude by describing a continuous analog and giving a characterization of all continuous solutions.

\section{Polynomial solutions}

The terms in the classical Galileo sequence \(a_n = 2n-1\) can be written as a difference of squares, \(a_n = n^2 - (n-1)^2\). Generalizing this, Zeitlin~\cite{Zeitlin} observed that for any integer \(d \ge 1\), the polynomial sequence \(a_n = n^d - (n-1)^d\) always satisfies the Galileo relation. We show that if \((a_n)\) is a Galileo sequence whose terms are generated by a polynomial, then it \emph{must} arise from such a difference of powers. Our first step is to prove the following lemma. 

\begin{lemma}\label{lem:dilation1}
Let $T(x)$ be a real polynomial, let $\alpha>0$ with $\alpha\ne1$, and let $\lambda\in\mathbb{R}$.
If
\[
T(\alpha x)=\lambda\,T(x),
\]
for all real $x$, then either $T\equiv0$ or else
\[
T(x)=C x^d,\qquad
\lambda=\alpha^d,
\]
for some constant $C\ne0$ and integer $d\ge0$.
\end{lemma}

\begin{proof}
Write $T(x)=\sum_{j=0}^m c_j x^j$ with $c_m\ne0$. Equating coefficients in $T(\alpha x)=\lambda T(x)$ gives $c_j \alpha^j=\lambda c_j$. If $c_j \neq 0$, then $\lambda = \alpha^j$. Since the values $\alpha^j$ are distinct, at most one coefficient $c_j$ can be nonzero. Hence $T(x)=Cx^d$.
\end{proof}

\begin{theorem}\label{thm:poly}
Let $(a_n)$ be a sequence satisfying the Galileo relation (G). Assume that there exists a real polynomial $p(x)$ such that $a_n = p(n)$. Then there exist a constant $C\ne0$ and an integer $d\ge1$ such that
\begin{align*}
a_n&= C\left(n^d-(n-1)^d\right),\\
S_n&= C\,n^d,
\end{align*}
and in particular $k = 2^d$.
\end{theorem}

\begin{proof}
Since $a_n=p(n)$, there exists a polynomial $P(x)$ (of degree one higher than $p$) such that
\[
S_n = \sum_{i=1}^n p(i) = P(n).
\]
The Galileo relation $S_{2n}=kS_n$ then reads
\[
P(2n) = k\,P(n). 
\]
Define $Q(x)=P(2x)-kP(x)$. Then $Q(n)=0$ for all $n\ge1$, so $Q\equiv0$, since a polynomial with infinitely many zeros is identically zero. Hence, for $x\in\mathbb{R}$,
\[
P(2x)=kP(x).
\]
We may now apply Lemma~\ref{lem:dilation1} with $T=P$, $\alpha=2$ and $\lambda=k$. Since $S_1=a_1>0$, the polynomial $P$ is not identically zero. Thus the lemma gives
\[
P(x) = C x^d,\qquad k = 2^d,
\]
for some constant $C\ne0$ and integer $d\ge0$. Because $k>1$ we have $d\ge1$. Finally, we obtain
\[
S_n = P(n) = C n^d,
\]
and so
\[
a_n = S_n - S_{n-1}
    = C\bigl(n^d - (n-1)^d\bigr).
\]
Thus within the class of polynomials, Galileo sequences can only be generated by this specific form.
\end{proof}

\section{Non-polynomial solutions}

The previous result raises a natural question: are there \emph{non-polynomial} sequences satisfying the Galileo condition, and if so, can we describe them all? The existence of such examples is demonstrated by Tattersall \cite[p.~23]{Tattersall}. Indeed, for every $k>3$ there exists a strictly increasing, non-polynomial sequence of positive integers satisfying this relation, given recursively by $a_1=1,a_2=k-1$ and
\begin{align*}
a_{2n-1} &= \biggl\lfloor \frac{ka_n-1}{2}\biggr\rfloor,\\
a_{2n}&= \biggl\lfloor \frac{ka_n}{2}\biggr\rfloor+1.
\end{align*}
Thus non-polynomial Galileo sequences certainly exist; we now show how both this example and the previous polynomial family fit into a single binary-tree structure. 

\section{A local identity and the binary tree viewpoint}

E\u{g}ecio\u{g}lu \cite{Egecioglu} noted that the global condition (G) on partial sums can be replaced by an equivalent local identity (L) involving triples $(a_n,a_{2n-1},a_{2n})$. For the sake of completeness, we prove that (L) and (G) are equivalent conditions.

\begin{lemma}\label{lem:equiv}
The Galileo relation (G) holds if and only if
\begin{equation}\label{(L)}
a_{2n-1}+a_{2n}=k\,a_n,\tag{L}
\end{equation}
for all $n\geq 1$. 
\end{lemma}

\begin{proof}
First suppose (G) holds. Then since $S_{2n}=kS_n$ and $S_{2n-2}=kS_{n-1}$, we obtain
\[
a_{2n-1}+a_{2n}
= S_{2n}-S_{2n-2}
= k(S_n-S_{n-1})
= k\,a_n.
\]
Now suppose (L) holds. Summing both sides of (L) from $n=1$ to $N$ gives
\begin{align*}
\sum_{n=1}^{N} (a_{2n-1}+a_{2n}) = \sum_{n=1}^{N} k\,a_n\Longrightarrow \sum_{n=1}^{2N} a_{n} = k\sum_{n=1}^{N} \,a_n.
\end{align*}
This shows that $S_{2N}=kS_N$, proving (G).
\end{proof}

We now demonstrate that identity (L) can be developed into a binary-tree representation of \emph{all} Galileo sequences. For $n=1$, the local identity gives
\(
a_1+a_2=ka_1,
\)
so once $a_1>0$ is chosen, the root relation forces
\(
a_2=(k-1)a_1.
\)
For each $n\ge2$, we may regard the node $n$ as having two children
\[
n \longrightarrow (2n-1,2n)
\]
whose labels add up to $k$ times the parent label. More explicitly, suppose $(a_n)$ is a positive Galileo sequence. For each $n\ge2$, dividing the local identity (L) by $a_n\neq 0$ gives
\(
\frac{a_{2n-1}}{a_n} + \frac{a_{2n}}{a_n} = k.
\)
It is therefore natural to introduce the ratios
\[
b_n = \frac{a_{2n-1}}{a_n} \text{ and }
c_n = \frac{a_{2n}}{a_n},
\]
so that $b_n+c_n=k$. We call $(b_n,c_n)$ \emph{splitting factors}. Thus, for each $n\ge2$, our recurrence takes the simpler form
\[
a_{2n-1}=b_n\,a_n \text{ and } a_{2n}=c_n\,a_n.
\]

\noindent Conversely, once $a_1>0$ is chosen and $a_2=(k-1)a_1$ is fixed, any choice of positive splitting factors $(b_n,c_n)$ satisfying $b_n+c_n=k$ for each $n\ge2$ uniquely determines a sequence $(a_n)$ by propagating the rule down the binary tree. Indeed, for $n=1$ we have
\(
a_1+a_2=a_1+(k-1)a_1=ka_1,
\)
and for each $n\ge2$ the defining relations give
\[
a_{2n-1}+a_{2n}=b_n a_n+c_n a_n=(b_n+c_n)a_n=ka_n.
\]
Thus the local identity (L) holds for all $n\ge1$, and Lemma~\ref{lem:equiv} shows that $(a_n)$ satisfies the Galileo relation~\eqref{G}. We record these findings in the following theorem.

\begin{theorem}\label{thm:tree}
Fix $k>1$ and choose $a_1>0$. Set
\(
a_2=(k-1)a_1.
\)
For each $n\ge2$, choose real numbers $b_n,c_n>0$ with $b_n+c_n=k$, and define
\[
a_{2n-1}=b_n\,a_n,\qquad a_{2n}=c_n\,a_n
\qquad(n\ge2).
\]
Then $(a_n)$ satisfies the Galileo relation~\eqref{G}. Conversely, every positive sequence satisfying~\eqref{G} arises uniquely in this way, with
\[
a_2=(k-1)a_1,\qquad
b_n=\frac{a_{2n-1}}{a_n},\qquad
c_n=\frac{a_{2n}}{a_n}
\qquad(n\ge2).
\]
\end{theorem}

Tracing the path from $2$ to any node $n\ge2$ in the binary tree yields an exact product representation. Consider the binary tree with root $2$ and children
\[
m\longrightarrow (2m-1,2m),
\]
where $m\geq 2$. If
\(
2=m_0\to m_1\to\cdots\to m_r=n
\)
is the unique path from $2$ to $n$, then
\[
a_n=a_2\prod_{j=0}^{r-1}\omega_j,
\]
where
\[
\omega_j=
\begin{cases}
b_{m_j}, & \text{if } m_{j+1}=2m_j-1,\\[1mm]
c_{m_j}, & \text{if } m_{j+1}=2m_j.
\end{cases}
\]
Since $a_2=(k-1)a_1$, this also gives $a_n$ in terms of $a_1$. The freedom to choose the factors in the above product yields a large family of non-polynomial solutions. 

\begin{example} [Equal children for $k=4$]
We illustrate the binary-tree representation in the case $k=4$.
Start with $a_1=1$. It follows that $a_2=(4-1)a_1=3$. Dividing by $a_1$ gives the normalized splitting factors at the root:
\[
b_1=\frac{a_1}{a_1}=1,\qquad
c_1=\frac{a_2}{a_1}=3,
\qquad b_1+c_1=4.
\]
\noindent From level $n\ge2$ onward we now choose the equal splitting factors $b_n=c_n=2$, so that $b_n+c_n=4$. Thus for every $n\ge2$,
\begin{align*}
a_{2n-1}=2a_n,\quad
a_{2n}=2a_n.
\end{align*}
We then apply the above recursion to give the binary tree below. The resulting monotone Galileo sequence is illustrated in Figure \ref{fig:balanced-tree} below. It is constant on each interval $(2^{m-1},2^m]$ for $m\ge2$: 
\[
(a_n)=1,\,3,\,6,\,6,\,12,\,12,\,12,\,12,\ldots
\]
In fact, for $n\ge2$, this sequence is given by an explicit formula
\(
a_n=3\cdot 2^{\lfloor \log_2(n-1)\rfloor}.
\) (See A053644 in the OEIS \cite{OEIS}, shifted by one index and multiplied by $3$.)

\begin{figure}[h]
\centering

\begin{tikzpicture}[
  level distance=12mm,
  sibling distance=8mm,
  edge from parent/.style={draw,-latex}
]
\Tree
[.$1$
    [.$3$
        \edge node[above left] {$b_2=2$};
        [.$6$
            \edge node[above left] {$b_3=2$};
            [.$12$ ]
            \edge node[above right] {$c_3=2$};
            [.$12$ ]
        ]
        \edge node[above right] {$c_2=2$};
        [.$6$
            \edge node[above left] {$b_4=2$};
            [.$12$ ]
            \edge node[above right] {$c_4=2$};
            [.$12$ ]
        ]
    ]
]
\end{tikzpicture}
\caption{Equal children: $a_1=1$, $a_2=3$ given, then $b_n=c_n=2$ for $n\ge2$.}
\label{fig:balanced-tree}
\end{figure}
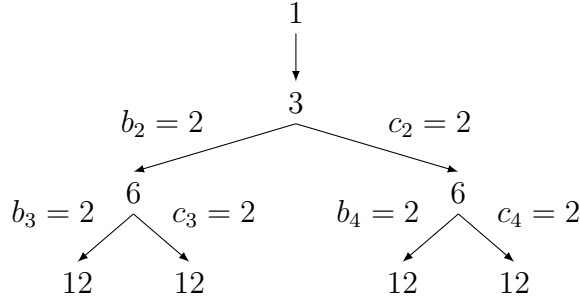
\end{example}

\begin{example} [Unequal children for $k=4$]
For contrast, set $b_n=1$, $c_n=3$ for every $n\ge2$, so each parent splits as $(1,3)$. The resulting tree (Figure~\ref{fig:unbalanced-tree}) still satisfies \textup{(G)}, but the labels listed in natural order are no longer monotone, nor are they constant on each interval $(2^{m-1},2^m]$ for $m\ge2$. The binary tree for this sequence is illustrated below in Figure \ref{fig:unbalanced-tree}.  The sequence begins as:
\[
(a_n)=1,\,3,\,3,\,9,\,3,\,9,\,9,\,27,\ldots
\]
This sequence also has an explicit formula,
\(
a_n=3^{s_2(n-1)},
\)
where $s_2(m)$ denotes the number of $1$'s in the binary expansion of $m$. (See A048883 in the OEIS \cite{OEIS}, shifted by one index.)

\begin{figure}[h]
\centering
\begin{tikzpicture}[
  level distance=12mm,
  sibling distance=8mm,
  edge from parent/.style={draw,-latex}
]
\Tree
[.$1$
    [.$3$
        \edge node[above left] {$b_2=1$};
        [.$3$
            \edge node[above left] {$b_3=1$};
            [.$3$ ]
            \edge node[above right] {$c_3=3$};
            [.$9$ ]
        ]
        \edge node[above right] {$c_2=3$};
        [.$9$
            \edge node[above left] {$b_4=1$};
            [.$9$ ]
            \edge node[above right] {$c_4=3$};
            [.$27$ ]
        ]
    ]
]
\end{tikzpicture}
\caption{Unequal children: $a_1=1$, $a_2=3$ given, then $(b_n,c_n)=(1,3)$ for $n\ge2$.}
\label{fig:unbalanced-tree}
\end{figure}
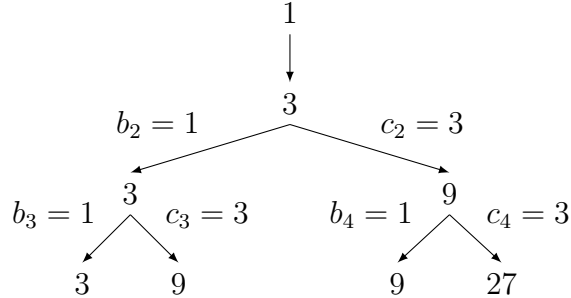

\end{example}

\noindent

\section{Important examples}

As every Galileo sequence can be generated in this fashion, we can recover each of the families encountered thus far by some suitable choice of $(b_n,c_n)$. 

\begin{example}[Galileo's sequence of odd integers]
For the classical sequence $a_n = 2n-1$ we have $S_n = n^2$ and hence $k=4$. The splitting factors are
\[
b_n=\frac{a_{2n-1}}{a_n}
    =\frac{4n-3}{2n-1},\qquad
c_n=\frac{a_{2n}}{a_n}
    =\frac{4n-1}{2n-1},
\]
which satisfy $b_n + c_n = 4$ for all $n\ge1$. This is the simplest instance of the polynomial family discussed in the next example.
\end{example}

\begin{example}[Polynomial sequences]
Let $k=2^d$ and
\[
a_n=C\bigl(n^d-(n-1)^d\bigr).
\]
Then the splitting factors are
\[
b_n=\frac{(2n-1)^d-(2n-2)^d}{n^d-(n-1)^d},\qquad
c_n=\frac{(2n)^d-(2n-1)^d}{n^d-(n-1)^d},
\]
which always satisfy $b_n+c_n=2^d=k$. These ratios describe all polynomial Galileo sequences. For \(C=1\) and fixed \(d\), the sequences \(a_n=n^d-(n-1)^d\) appear as columns of A047969 in the OEIS \cite{OEIS}; for example, \(d=1,2,3,4,5\) correspond respectively to A000012, A005408, A003215, A005917 (shifted by one index), and A022521.
\end{example}

\begin{example}[Tattersall's sequence]
Fix $k>3$ and define
\[
a_{2n-1}=\Bigl\lfloor\frac{ka_n-1}{2}\Bigr\rfloor,\quad
a_{2n}=\Bigl\lfloor\frac{ka_n}{2}\Bigr\rfloor+1.
\]
The corresponding splitting factors are
\[
b_n=\frac{\lfloor(ka_n-1)/2\rfloor}{a_n},\quad
c_n=\frac{\lfloor ka_n/2\rfloor+1}{a_n},
\]
with $b_n+c_n=k$ for all $n$. 

These produce strictly increasing, non-polynomial Galileo sequences. For comparison with the OEIS, note that the OEIS family is parameterized so that \(a_2=k\), whereas here \(a_2=k-1\). Thus our cases \(k=4,5,6\) correspond to the OEIS entries A005408, A385587, and A385643; see also the family entry A385610 and its cross-references \cite{OEIS}.

\end{example}

\section{Monotone Galileo sequences of integers}

In the preceding section we saw that the binary-tree representation of Galileo sequences allows for an abundance of non-polynomial examples. It is therefore natural to ask what happens when we restrict attention to positive, integer-valued sequences that grow monotonically. We begin with the special cases $k=2$ and $k=3$. We then prove that every monotone sequence exhibits power-law growth. 

\begin{theorem}\label{cor:k2-k3}
Let $(a_n)$ be a positive, non-decreasing, integer-valued sequence satisfying the Galileo relation \textup{(G)} with integer $k>1$. If $k=2$, then $(a_n)$ is constant. If $k=3$, no such sequence exists.
\end{theorem}
\begin{proof}
First suppose $k=2$. Then the root relation gives
\(
a_2=(k-1)a_1=a_1.
\)
For each $n\ge2$, by Theorem~\ref{thm:tree} we have
\[
a_{2n-1}=b_n a_n,\qquad a_{2n}=c_n a_n,\qquad b_n+c_n=2.
\]
Monotonicity gives $a_{2n-1}\ge a_n$ and $a_{2n}\ge a_n$, so $b_n,c_n\ge1$,
hence $b_n=c_n=1$. Thus $a_{2n-1}=a_{2n}=a_n$ for all $n\ge2$, and since also $a_2=a_1$, induction yields $a_n=a_1$ for all $n\ge1$.

\medskip

Now suppose $k=3$, so that $a_{2n-1}+a_{2n}=3a_n$. Assume that there exists a monotone Galileo sequence of integers. Define
\[d=\min\{a_{n+1}-a_n:n\in \mathbb{N}\}\]
to be the minimum difference between consecutive terms in this sequence. By assumption, $d$ is a non-negative integer. Suppose $a_{i+1}-a_i=d$ attains this minimum. Then by minimality,
\[
d \le
\begin{cases}
a_{2i} - a_{2i-1} \\
a_{2i+1} - a_{2i} \\
a_{2i+2} - a_{2i+1}. 
\end{cases}
\]
Combining these gives
\begin{align*}
4d &\leq (a_{2i+2}-a_{2i+1})+(a_{2i+1}-a_{2i}) +(a_{2i+1}-a_{2i})+(a_{2i}-a_{2i-1}) \\
&= (a_{2i+2}+a_{2i+1})-(a_{2i}+a_{2i-1}) \\
&= 3a_{i+1} - 3a_i \\
&= 3d. 
\end{align*}
Hence $d=0$, and equality holds throughout. Thus
\[
a_{2i-1}=a_{2i}=a_{2i+1}=a_{2i+2}.
\]
Applying the Galileo relation, $a_{2i-1}+a_{2i}=3a_i$, we obtain
\[
2a_{2i} = 3a_i \implies a_{2i} = \frac{3}{2}a_i,
\]
and hence all four terms equal $\frac{3}{2}a_i$. Repeating this argument with $i$ replaced by $2i, 4i, 8i,\ldots$ and so forth, we inductively obtain
\[a_{2^ki}=a_{2^ki+1}= \left( \frac{3}{2} \right)^ka_i.\]
Finally, as $a_i\in\mathbb{N}$, this eventually produces a non-integer, yielding the desired contradiction. Thus no positive non-decreasing integer solution exists for $k=3$.
\end{proof}

The case $k=3$ is exceptional. It is the only value for which we cannot generate a monotone Galileo sequence of integers. On the other hand, if $k\geq 4$, we can construct uncountably many positive non-decreasing integer Galileo sequences. We leave these details to the curious reader. Despite the abundance of such sequences, we can still give bounds for how quickly they can grow.

\begin{theorem}
\label{prop:asymptotic}
Let $(a_n)$ be positive, non-decreasing, and integer-valued, and satisfy
\textup{(G)} with constant $k>1$. Set $d=\log_2 k$. 
\begin{enumerate}[label=\textup{(\alph*)}]
\item  If $n=2^r m$ with $m$ odd, then $S_n = c(m)\,n^d$ where $c(m)=\frac{S_m}{m^d}$.
\item  There exist constants $C_1,C_2>0$ such that $C_1 n^d \le S_n \le C_2 n^d$.
\item  Consequently, there exist $D_1,D_2>0$ such that $D_1 n^{d-1} \le a_n \le D_2 n^{d-1}.$
\end{enumerate}
\end{theorem}

\begin{proof}
\leavevmode
\begin{enumerate}[label=\textup{(\alph*)}]

\item Write $T(n)=\frac{S_n}{n^d}$. Then
\[
T(2n)=\frac{S_{2n}}{(2n)^d}
=\frac{kS_n}{2^d n^d}
=T(n),
\]
since \(k=2^d\). Hence, if \(n=2^r m\) with \(m\) odd, repeated application of this identity gives
\[
T(n)=T(2^r m)=T(m).
\]
Therefore
\[
\frac{S_n}{n^d}=\frac{S_m}{m^d},
\]
so
\(
S_n=c(m)n^d,
\)
where
\(c(m)=\frac{S_m}{m^d}\).

\item Fix \(n\ge 1\). Choose \(r\) such that
\[
2^r \le n < 2^{r+1}.
\]
Since \((a_n)\) is positive, the partial sums \((S_n)\) are increasing, so
\[
S_{2^r} \le S_n \le S_{2^{r+1}}.
\]
Iterating \textup{(G)} yields
\begin{align*}
S_{2^r}&=S_1(2^r)^d,\\
S_{2^{r+1}}&=S_1(2^{r+1})^d.
\end{align*}
Since \(2^r\ge n/2\) and \(2^{r+1}\le 2n\), it follows that
\[
\frac{S_1}{2^d}n^d \le S_n \le S_1 2^d\,n^d.
\]
Thus we may take \(C_1=\frac{S_1}{2^d}\) and \(C_2=S_1 2^d\).

\item From \textup{(G)},
\[
\sum_{i=n+1}^{2n} a_i = S_{2n}-S_n=(k-1)S_n.
\]
Since \((a_n)\) is non-decreasing,
\[
n a_n \le \sum_{i=n+1}^{2n} a_i \le n a_{2n}.
\]
Using part \textup{(b)}, we obtain
\[
(k-1)C_1 n^d \le (k-1)S_n = \sum_{i=n+1}^{2n} a_i \le (k-1)C_2 n^d.
\]
Hence
\begin{align*}
a_n &\le (k-1)C_2\,n^{d-1},\\
a_{2n} &\ge (k-1)C_1\,n^{d-1}.
\end{align*}
So we may take \(D_2=(k-1)C_2\). For the lower bound, first let \(m=2n\) be even. Then
\[
a_m=a_{2n}\ge (k-1)C_1\,n^{d-1}
      =(k-1)C_1\,2^{\,1-d}m^{d-1}.
\]

Now let \(m=2n+1\ge 3\) be odd. By monotonicity,
\[
a_m\ge a_{m-1}=a_{2n}\ge (k-1)C_1\,n^{d-1}.
\]
Since \(n=(m-1)/2\ge m/4\), and since \(d\ge 1\), we get
\[
a_m\ge (k-1)C_1\left(\frac m4\right)^{d-1}
    =(k-1)C_1\,4^{\,1-d}m^{d-1}.
\]

Therefore, for all \(m\ge 2\),
\[
a_m \ge (k-1)C_1\,4^{\,1-d}m^{d-1}.
\]
Finally, by decreasing this constant if necessary to account for \(m=1\), we obtain a constant \(D_1>0\) such that
\(
a_m\ge D_1 m^{d-1}
\)
for all \(m\ge 1\), as required.
\end{enumerate}
\end{proof}

\section{A continuous analog}

It is a natural extension to consider the continuous analog of Galileo sequences, which we call \emph{Galileo functions}. This amounts to replacing partial sums by integrals.

\begin{definition}
Suppose $a>0$ and $a\neq 1$. A continuous function $f:[0,\infty)\to[0,\infty)$ is a Galileo function if
\begin{equation}\label{G'}
\int_{0}^{ax} f(t) \,dt = b\int_{0}^{x} f(t) \,dt,
\tag{G'}
\end{equation}
for some positive reals $a$ and $b$, for all positive real $x$.
\end{definition}

We show that Galileo functions are exactly power laws multiplied by a continuous periodic function of $\log_a x$. 

\begin{theorem}
Assume \(a>0\) with \(a\neq 1\). A Galileo function \(f\) must satisfy
\[
f(x)=g(\log_a x)\,x^{\log_a(b/a)}
\qquad (x>0),
\]
for some continuous function \(g:\mathbb R\to[0,\infty)\) of period \(1\).
Conversely, any function of this form is a Galileo function.
\end{theorem}

\begin{proof}
As $f$ is continuous, we can differentiate both sides of \eqref{G'} to give
\begin{equation*}
af(ax)=bf(x).
\end{equation*}
Define $g:\mathbb{R}\to[0,\infty)$ by
\begin{equation*}
g(x)=\frac{f(a^x)}{(b/a)^x}.
\end{equation*}
Since $f$ is continuous, and the functions $x\mapsto a^x$ and $x\mapsto (b/a)^x$ are continuous, $g$ is continuous. We observe:
\[
g(x+1)=\frac{f(a\cdot a^x)}{(b/a)\cdot (b/a)^x}
=\frac{(b/a)f(a^x)}{(b/a)\cdot (b/a)^x}
=g(x).
\]
So $g$ has period $1$. Reversing the definition of $g$, we conclude that
\[
f(x)=g(\log_a x)\cdot x^{\log_a(b/a)}
\qquad (x>0),
\]
for some continuous function $g$ of period $1$. The converse is straightforward and is left to the reader.
\end{proof}

\section{Disclosure on AI-assisted tools}

During preparation of this manuscript, ChatGPT was used to flag possible issues of style-guide compliance and to suggest possible OEIS entries and related references. All such suggestions were manually checked by the authors. 

\section{OEIS sequences}

The OEIS sequences appearing in this paper are A000012, A003215, A005408, A005917, A022521, A047969, A048883, A053644, A385587, A385610, and A385643.

\end{document}